\documentclass[reqno,10pt]{amsart}
\usepackage{amsmath,amsthm,amsfonts,amssymb,amscd, bbm, xcolor}

\theoremstyle{plain}
\newtheorem{maintheorem}{Theorem}

\newtheorem{maincorollary}[maintheorem]{Corollary}

\newtheorem{theorem}{Theorem }[section]
\newtheorem{proposition}[theorem]{Proposition}
\newtheorem{lemma}[theorem]{Lemma}
\newtheorem{corollary}[theorem]{Corollary}
\newtheorem{claim}[theorem]{Claim}
\theoremstyle{definition}

\newtheorem{definition}[theorem]{Definition}

\newcommand{\D}{\ensuremath{\mathbb{D}}}
\newcommand{\R}{\ensuremath{\mathbb{R}}}

\newcommand{\Z}{\ensuremath{\mathbb{Z}}}
\newcommand{\N}{\ensuremath{\mathbb{N}}}
\newcommand{\T}{\ensuremath{\mathbb{T}}}

\newcommand{\Jac}{\operatorname{Jac}}

\newcommand{\ds}{\displaystyle}

\newcommand{\ba}{{\textbf{a}}}
\newcommand{\bb}{{\textbf{b}}}
\newcommand{\bc}{{\textbf{c}}}
\newcommand{\bv}{{\textbf{v}}}
\newcommand{\bu}{{\textbf{u}}}
\newcommand{\bt}{{\textbf{t}}}
\newcommand{\sC}{C}
\newcommand{\sE}{E}

\newcommand{\diam}{\operatorname{diam}}

\newcommand{\torus}{{\mathbb{T}^u}}

\newcommand{\fm}{\mathfrak{m}}
\newcommand{\fpsi}{\psi}

\begin{document}
\title[Higher-dimensional Attractors with ACIP]{Higher-dimensional Attractors with absolutely continuous invariant probability}
	
\date{}
	
\author{Carlos Bocker}
\address[Carlos Bocker]{Department of Mathematics, UFPB\\ Jo\~ao Pessoa-PB, Brazil}
\email{cbocker@gmail.com}

\author{Ricardo T. Bortolotti}
\address[Ricardo T. Bortolotti]{Department of Mathematics, UFPE\\ Recife-PE, Brazil}
\email{ricardo@dmat.ufpe.br}
	
\begin{abstract}
Consider a dynamical system $T:\torus\times \R^{d} \rightarrow \torus\times \R^{d} $ given by $ T(x,y) = (\sE (x), \sC(y) + f(x) )$, where $\sE$ is a linear expanding map of $\torus$, $\sC$ is a linear contracting map of $\R^d$ and $f$ is in $C^2(\torus,\R^d)$. We prove that if $T$ is volume expanding and $u\geq d$, then for every $\sE$ there exists an open set $\mathcal{U}$  of pairs $(\sC,f)$ for which the corresponding dynamic $T$  admits an absolutely continuous invariant probability.
		
A geometrical characteristic of transversality between self-intersections of images of $\torus\times\{ 0 \}$ is present in the dynamic of the maps in $\mathcal{U}$. In addition, we give a condition between $\sE$ and $\sC$ under which it is possible to perturb $f$ to obtain a pair $(\sC,\tilde{f})$ in $\mathcal{U}$. 
\end{abstract}
	
\maketitle

\section{Introduction}
	
In the study of ergodic theory, hyperbolic attractors for smooth invertible maps admits only singular invariant measures with respect to volume, since they have zero Lebesgue measure \cite{bowen}. For non-invertible maps, absolutely continuous invariant measures do exist and are usually associated with the positivity of the Lyapunov exponents (see \cite{Alves, keller, solano}).
	
There are attractors on surfaces with one negative Lyapunov exponent that admit absolutely continuous invariant probabilities. Fat backer maps \cite{fat.baker.map} and fat solenoidal attractors \cite{tsujii.ssd, tsujii.fsa} are examples of such maps.
	
	
	
In the present work, we study higher-dimensional attractors with $d$ negative Lyapunov exponents that admit absolutely continuous invariant measures.
	
We consider volume expanding skew-products $T:\torus\times \R^{d} \rightarrow \torus\times \R^{d} $ given by
$$ T(x,y) = \big ( \sE(x) , \sC(y) + f(x)  \big)$$
where $\sE:\torus\rightarrow\torus$ is an expanding map, induced by a linear map $E:\R^u \rightarrow \R^u$ that preserves the lattice $\Z^u$, $\torus = \R^u / \Z^u$, $C$ a linear contracting map of $\R^d$ and $f$ is a $C^{2}$ function of $\torus$ into $\R^{d}$.
	
Considering some constant $K>0$ for which $M= \torus \times [-K,K]^d$ satisfies $T(M) \subset M$ we get the attractor $\Lambda=\cap_{n\geq 0}{T^n(M)}$. Since the  restriction of $T$ to $\Lambda$ is a transitive hyperbolic endomorphism, it admits a unique SRB measure $\mu_T$ supported on $\Lambda$. The goal of this work is to give conditions that guarantee the absolute continuity of $\mu_T$ with respect to the volume measure of $\torus \times \R^d$.
	
Let $E(u)$ be the set of linear expanding maps of $\torus$ and $C(d)$ the set of linear contractions of $\R^d$. Denoting $T=T(\sE,\sC,f)$,
the first result in this work is:
	
\begin{maintheorem}\label{teo.a}
Given integers $u \geq  d $ and an expanding map $\sE\in E(u)$, there exists a nonempty open subset $\mathcal{U}$ of $C(d) \times C^2(\torus, \R^d)$ such that the corresponding SRB measure $\mu_{T}$ of every map $T=T(\sE,\sC,f)$ for $(\sC,f)\in\mathcal{U}$ is absolutely continuous with respect to the volume of $\torus\times\R^d$.
\end{maintheorem}
	
Two features of non-invertible maps are important along the proof of Theorem A in order to obtain absolute continuity of the invariant measure: volume-expansion of the dynamic and transversal overlaps between the images of subsets of the domain (see Theorem \ref{teo.d}).
	
Given $\sE\in E(u)$, let us consider the following subset $C(d;E)$ of $C(d)$:
$$\displaystyle C(d;\sE)=\left\{ \sC \in C(d) , |\det \sC| > |\det \sE|^{-1} \text{ and } \|\sC\| < \frac{\|\sE^{-1}\|^{-1}}{|\det \sE|^{\frac{1}{{u-d+1}}} } \right\}.$$
	
We prove that absolute continuity of $\mu_T$ is generic for $T=T(\sE,\sC,f)$ with $(\sC,f)\in C(d;\sE)\times C^2(\torus, \R^d)$. More precisely, given a finite family of perturbations $\phi_1,\cdots,\phi_s$, for each $\bt=(t_1,\cdots,t_s)\in \R^s$ we consider the function $f_{\bt} = f + t_1 \phi_1 + \cdots + t_s \phi_s$ and the corresponding dynamic
$T_{\bt}(x,y)=(\sE(x),\sC(y)  + f_{\bt}(x))$.
	
\begin{maintheorem}\label{teo.b}
Given integers $u \geq d $, $\sE\in E(u)$, $\sC\in C(d;E)$, there exists functions $\phi_k \in C^{\infty}(\torus,\R^d)$, $k=1,2,\dots,s$,
such that for any $f\in C^{2}(\torus,\R^d)$
the set of parameters $\bt=(t_1,\dots,t_s)$
for which the corresponding SRB measure $\mu_{T_{\bt}}$ is absolutely continuous
has full Lebesgue measure.
\end{maintheorem}
	
Since the maps $\sC$ in $C(1)$ are multiplications by a scalar $\lambda$, when $\sE$ is multiple of the identity the set $C(d;\sE)$ is exactly the set $(|\det \sE|^{-1},1)$. So Theorem \ref{teo.b} has the following consequence.
	
\begin{maincorollary}\label{cor.c}
Given an integer $u \geq 1$, a linear expanding map $\sE\in E(u)$ that is multiple of the identity, $\lambda \in (|\det \sE|^{-1},1)$ and $\sC:\R \to \R$ given by $\sC(y)=\lambda y$,  there exists functions $\phi_k \in C^{\infty}(\torus,\R)$, $k=1,2,\dots,s$,
such that for any $f\in C^{2}(\torus,\R)$
the set of $\bt=(t_1,\dots,t_s)$
for which the corresponding SRB measure $\mu_{T_t}$ is absolutely continuous
has full Lebesgue measure.
\end{maincorollary}
	
It is interesting to note that the condition $\lambda \in (|\det \sE|^{-1},1)$ includes values of $\lambda$ arbitrarily close to $1$. This condition on $\lambda$  is the optimal one in order to expect absolute continuity of the invariant measure, because if $\lambda < |\det \sE|^{-1}$ then the map $T$ is volume contracting, which implies that the attractor $\Lambda$ has null volume and that every invariant measure supported in $\Lambda$ must be singular.
	
The proofs of Theorems \ref{teo.a} and \ref{teo.b} are based on an idea of transversality between unstable manifolds, similar to the one in \cite{tsujii.fsa}. The set of maps $\mathcal{U}$ is exactly the set of dynamics $T$ satisfying such condition of transversality (see Definitions \ref{transversality.1} and \ref{transversality.2}). We will see that the transversality condition implies the absolute continuity of $\mu_T$ and is generic under the assumptions of Theorem \ref{teo.b}. The case $u=d=1$ is the one treated in \cite{tsujii.fsa}.

This method of transversality was introduced in the dynamical systems for surface endomorphisms in \cite{tsujii.fsa, tsujii.acta}. A similar idea can be applied to partial hyperbolic dynamical systems in order to prove the existence and finiteness of physical measures \cite{bortolotti, tsujii.acta}, what is conjectured to be valid for a typical dynamical system \cite{palis}.
	
In Section 2, we give the definitions, including the transversality condition, and the statements of this work. In Section 3, we prove that the transversality condition implies the absolute continuity of the SRB measure, giving estimates on the  $L^2$-regularity of this measure. In Section 4, we prove that the transversality condition is generic under the assumption that $\sC\in C(d;\sE)$.
		
\section{Definitions and statements}
	
Let us fix some notations evolving the partition of the basis that codify the action of the expanding map $\sE$ in $\torus$. Codifying this dynamic will be important to define the transversality condition and the generic families of perturbations.
	
Given integers $u$ and $d$, we consider the
dynamic $T=T(\sE,\sC,f):\torus\times \R^{d} \rightarrow \torus\times \R^{d}$ given by
\begin{equation}
T(x,y) = \big ( \sE(x) , \sC(y)  + f(x)  \big)\,,
\end{equation}
where $\sE\in E(u)$ is an expanding map whose lift $E:\R^u \rightarrow \R^u$ is a linear map preserving the lattice $\Z^u$, $\sC\in C(d)$ is a linear contracting map, that is, $\|\sC(v)\|<\|v\|$ for every $v\in\R^d$  and $f \in C^2(\torus,\R^d)$. We suppose in the whole text that $T$ is volume expanding, which means that $E$ and $\sC$ satisfy $|\det \sE  \det \sC|>1$. If $T$ is not volume expanding then the attractor has zero volume and supports no absolute continuous invariant measure.
	
Let $\mathcal{R}=\{\mathcal{R}(1),\cdots,\mathcal{R}(r)\}$ be a fixed Markov partition for $\sE$, that is,
$\mathcal{R}(i)$ are disjoint open sets, the interior of each $\overline{R(i)}$ coincides with $R(i)$,
$\sE_{|_{\mathcal{R}(i)}}$ is one-to-one,
${\bigcup}_i \overline{\mathcal{R}(i)}= \torus$  and
$\sE({\mathcal{R}(i)}) \cap {\mathcal{R}(j)} \neq \emptyset$ implies that ${\mathcal{R}(j)}\subset {\mathcal{R}(i)}$.
It is a well-known fact that Markov partitions always exists for expanding maps (see \cite{markov.partition} for example).
	
Let us suppose that $\operatorname{diam}(\mathcal{R}) < \gamma$, where $\gamma>0$ is a constant such that:   for every $x\in \torus$ and $y\in \sE^{-1}(x)$ there exists a unique affine inverse branch $g_{y,x}:B(x,\gamma) \to \torus$ such that $g_{y,x}(x)=y$ and $\sE(g_{y,x}(z))=z$ for every $z \in B(x,\gamma)$.
	
Consider the set $\overline{I}=\{ 1, \cdots, r\}$ and $\overline{I}^n$ the set of words of length $n$ with letters in $\overline{I}$, $1\leq n \leq \infty$. Denoting by $\ba =(a_{i})_{i=1}^{n}$ a word in $\overline{I}^{n}$, define $I^n$ the subset of words $\ba =(a_{i})_{i=1}^{n}$ with the property that
$$\sE(\mathcal{R}(a_{i+1})) \cap \mathcal{R}(a_{i}) \neq \emptyset \text{ for every } 0\leq i \leq n-1 \,.$$
	
Consider the partition $\mathcal{R}^{n}:= \vee_{i=0}^{n-1} \sE^{-i}(\mathcal{R})$ and, for every $\ba \in \overline{I}^n$,  the sets $\mathcal{R}(\ba) = \cap_{i=0}^{n-1} \sE^{-i}(\mathcal{R}(a_{n-i}))$ in $ \mathcal{R}^n$, which are nonempty if and only if $\ba\in I^n$. The~truncation of $\ba=(a_j)_{j=1}^{n}$ to length $1\le p\le n$  is denoted by $[\ba]_p=(a_j)_{j=1}^{p}$.
	
For any $x\in\torus$, let us fix some $\pi(x) \in \overline{I}$ such that $x\in \overline{\mathcal{R}(\pi(x))}$  (it is unique for almost every $x\in\torus$).
For any $\bc\in I^p$, $1\leq p < \infty$, we consider $I^n(\bc)$ the set of words $\ba\in I^n$ such that $\sE^n(\mathcal{R}(\ba)) \cap \mathcal{R}(\bc) \neq \emptyset$.
Define $I^n(x):=I^n(\pi(x))$,
for $\ba\in I^{n}(x)$, denote by $\ba(x)$ the point $y\in\mathcal{R}(\ba)$ that satisfies $\sE^{n}(y)=x$.
	
For any $\ba\in I^n$ and $1\leq n <\infty$ we consider the set $\D(\ba):=\{x\in\torus | \ba \in I^n(x)\}$, which is a union of rectangles of the Markov partition. The image of  $\mathcal{R}(\ba)\times\{ 0 \}$ by $T^{n}$ is the graph of the function $S(\cdot,\ba):\D(\ba) \to \R^d$ given by
\begin{equation}
S(x,\ba) := \sum_{i=1}^{n} \sC^{i-1} f (\sE^{n-i}(\ba(x))) = \sum_{i=1}^{n} \sC^{i-1} f([\ba]_{i} (x))
\end{equation}
	
Consider the sets $I^\infty(x) = \{ \ba \in I^\infty$ such that $[\ba]_i \in I^i(x)$ for every $i\geq 1 \}$ and $\D(\ba):=\{ x \in \torus | \ba\in I^\infty(x)\}$ for $\ba\in I^\infty$.
If $\ba\in I^{\infty}(x)$, we define $S(x,\ba) = {\lim}_{n\to\infty} S(x,[\ba]_n)$.
The restriction of $S(\cdot,\ba)$ to each atom of the partition $\mathcal{R}$ is uniformly bounded in the $C^2$-topology, so it can be extended to the closure as a $C^{2}$ function redefining it on the border, which we denote by $S_{\bc}(\cdot,\ba)$.
	
In the first part of the work, we define the transversality condition that implies the absolute continuity of the SRB measure of the dynamic $T$, this transversality will be defined evolving the smallest singular value of the difference between two linear maps.
	
Given a linear map $A:\R^{u}\rightarrow \R^{d}$, denote by $$\displaystyle \fm(A):= \sup_{\dim W = d} \inf_{\|v\|=1, v\in W} {\|A(v)\|}$$ the smallest singular value of $A$. Consider the constants $\underline{\mu}= \|\sE^{-1}\|^{-1}$, $\overline{\mu}= \|\sE\|$, $\underline{\lambda}= {\|\sC^{-1}\|^{-1}}$, $\overline{\lambda}= \|\sC\|$, which are the minimum and maximum rates of expansion (or contraction) of $\sE$ (or $\sC$). Consider $N=|\det \sE|$ the degree of the expanding map, $J=|\det \sE  \det \sC |$ the Jacobian of $T$. Consider also $\alpha_{0} = \frac{\|f \|_{C^2}}{1-\|\sC\|}$ and $\theta=\overline{\lambda}\underline{\mu}^{-1}$. 
	
\begin{definition}\label{transversality.1}
Given $T=T(\sE,\sC,f)$ as above, integers $1\leq p, q <\infty$, $\bc\in I^p$ and $\ba, \bb \in I^{q}(\bc)$, we say that $\ba$ and $\bb$  are \textbf{transversal} on $\bc$ if
\begin{equation}\label{transversality}
\fm(   D S_{\bc}(x,\ba) - D S_{\bc}(y,\bb)    )   > 3 \theta^q \alpha_{0}
\end{equation}
for every $x, y \in \overline{\mathcal{R}(\bc)}$.
\end{definition}
	
In order to obtain the absolute continuity of the SRB measure, it will be used a condition of transversality between the graphs $S(x,\ba)$ for a big amount of pairs of $\ba$'s.
		
\begin{definition}\label{transversality.2}
Given $T$ as above, define the integer $\tau(q)$ by
\begin{equation}
\tau(q) =  \min_{p\geq 1} \max_{\bc\in I^{p}} \max_{\ba\in I^{q}(\bc)} \#\{ \bb\in I^{q}(\bc) | \text{ } \ba \text{ is not transversal to } \bb \text{ on } \bc \} \,.
\end{equation}

 We say that it holds the \textbf{transversality condition} if for some integer $q\in \N$ we have $\tau(q)<J^q$.
\end{definition}
	
The transversality condition means that for each $q$ and $\ba \in I^{\infty}$ the number of $\bb$'s that are not transversal to $[\ba]_ q$ may increase with $q$ at most at a rate smaller than $J^q$ (note that $1\leq \tau(q)\leq N^q$). This condition is used to estimate a regularity of the SRB measure, which will give its absolute continuity. The main step in the proof of Theorem A corresponds to the following Theorem.
	
\begin{maintheorem}\label{teo.d}
Given $T=T(\sE,\sC,f)$ satisfying the transversality condition, there exists a neighborhood $U\subset C(d;\sE)\times C^2(\torus,\R^d)$ of $(\sC,f)$ such that for every $(\tilde{\sC} ,\tilde{f} )\in U$ the corresponding SRB measure $\mu_{\tilde{T}}$ of $\tilde{T}=T(\sE,\tilde{\sC},\tilde{f})$ is absolutely continuous with respect to the volume of $\torus\times\R^d$ and its respective density is in $L^2(\torus\times \R^d)$.
\end{maintheorem}
	
For every $u \geq d$, we also verify that there exists dynamic $T=T(\sE,\sC,f)$ with $\sC\in C(d)$ that satisfies the transversality condition, even if $\sC \notin C(d,\sE)$. This is due to the following:
	
\begin{proposition}\label{non.empty}
Given $u\geq d$, integers $\mu_1,\cdots,\mu_u \geq 1$, $u_1, \cdots, u_d$ such that $u_1+\cdots+u_d=u$ and real numbers $\lambda_i \in \big(\mu_i^{-u_i},1\big)$. Consider $\torus=\mathbb{T}^{u_1}\times\cdots\times\mathbb{T}^{u_d}$,  the expanding map $\sE:\T^u \to\T^u$, $\sE(x_1,\cdots,x_d)=(\mu_1 x_1,\cdots,\mu_d x_d)$, and the contracting map $\sC:\R^d\to\R^d$, $\sC(y_1,\cdots,y_d)=(\lambda_1y_1,\cdots,\lambda_dy_d)$.
Then there exists a function $f  \in C^2(\torus, \R^d)$ for which the corresponding map $T=T(\sE,\sC,f)$ satisfies the transversality condition.
\end{proposition}
	
		
In the last part of the work, we deal with families of perturbations that give a dynamic satisfying the transversality condition.
	
Given $f_0 \in C^2(\torus, \R^d)$ and functions $\phi_1,\cdots, \phi_s \in C^\infty(\torus, \R^d)$, for any $s$-uple of parameters $\bt=(t_1,\cdots,t_s)$, we will consider the corresponding function
$ f_{\bt}(x) = f_0(x) + \sum_{k=1}^{m} t_{k} \phi_{k} (x) $ 
and the corresponding dynamic $T_\bt= T(\sE,\sC,f_\bt)$. For words $\ba\in I^{n}(x)$, $1\leq n\leq \infty$, we have also the corresponding $S(x,\ba;\bt) $. 
	
For a point $x \in\torus$ and a sequence $\sigma = (\ba_{0}, \ba_{1}, \cdots, \ba_{k})$ of words in $I^{\infty}(x)$, we consider the affine map $\fpsi_{x,\sigma}:\R^{s}\rightarrow (\R^{ud})^k$ defined by
$$\fpsi_{x,\sigma}(\bt) = \Big(  D S(x,\ba_{i};\bt) - D S(x,\ba_{0};\bt)  \Big)_{i=1,\cdots,k} $$
Each entry of $\fpsi_{x,\sigma}$ corresponds to the difference between the images of $T_{\bt}^q$ restricted to $\mathcal{R}(\ba_i)\times \{0\}$ and to $\mathcal{R}(\ba_0)\times \{0\}$.

Let us denote by $\Jac \fpsi_{x,\sigma}(\bt)$ the supremum of the Jacobian of the restrictions of $\fpsi_{x,\sigma}$ to $k$-dimensional subspaces: $\Jac(\fpsi_{x,\sigma})(\bt) = \sup_{dim L=k} {|\det D\fpsi_{x,\sigma} (\bt) |_{L} | } $.

The following definition is important in the proof of Theorem \ref{teo.b}, because corresponds to the kind of family $T_\bt$ that we shall construct in order to prove that the transversality condition is generic.
	
\begin{definition}\label{generic.family}
For every integer $n\geq 1$, we say that the family $T_{\bt}$ is \textbf{$n$-generic} on $A\subset\torus$ if the following holds: for any $x\in A$, for every $D\geq n^3$ and any sequence $(\ba_{0},\ba_{1},\cdots,\ba_{D})$ of $D$ words in $I^{\infty}(x)$ such that $[\ba_{i}]_{n}$ are distinct, taking $\kappa=\lfloor \frac{D}{2n}\rfloor$, there exists a subsequence $\sigma=(\bb_{0}, \bb_{1}, \cdots, \bb_{\kappa})$ of  $\{\ba_{i} \}$ of length  $\kappa + 1 $  such that $\bb_{0}=\ba_{0}$ and $\Jac(\fpsi_{x,\sigma})(\bt)>\frac{1}{2}$ for every $\bt\in \R^{m}$.
\end{definition}

The proof of Theorem \ref{teo.b} is divided into 2 parts: one corresponds to the construction of  functions $\phi_1,\cdots,\phi_s$ for which the family $T_\bt^n$ is generic for some large value of $n$ (Proposition \ref{generic.families}) and the other to check that the transversality condition is valid for almost every parameter $\bt$ for a certain generic family (Proposition \ref{zero.measure}). Then, Theorem \ref{teo.b} follows as consequence of Propositions \ref{generic.families} and \ref{zero.measure}.
	
\begin{proposition}\label{generic.families}
Given $u\geq d$, $\sE\in E(u)$, $\sC\in C(d;\sE)$, there exists an integer $n_0$ such that for every $n \geq n_0$ there exists functions $\phi_k \in C^{\infty}(\torus,\R^d)$, $1\leq k \leq s$, such that for every $f_0 \in C^2(\torus, \R^d)$  the corresponding family $T_{\bt}^{n}$ is $n$-generic on $\torus$.
\end{proposition}
	
For fixed $\sE,\sC$ define the set
$$\mathcal{T} = \left\{ f \in C^2(\torus,\R^d) , \underset{q\rightarrow \infty}{\limsup} \frac{\log \tau(q)}{q} \geq \log J \right\}\,.$$
If $f$ is not in $\mathcal{T}$, then the dynamic $T=T(\sE,\sC,f)$ satisfies the transversality condition and there exists some integer $q_0$ such that  $\tau(q) <J^q$ for every $q\geq q_0$.
	
The construction of generic families will be used to verify that for every $f_0\in C^2(\torus,\R^d)$ there exists a finite dimensional subspace $H \subset C^2(\torus,\R^d)$ such that the transversality condition is valid for the dynamic corresponding to almost every $g\in f_0+H$.
	
\begin{proposition}\label{zero.measure}
Given $u\geq d$, $\sE\in E(u)$ and $\sC\in C(d;\sE)$, there exists functions $\phi_k \in C^{\infty}(\torus,\R^d)$, $1\leq k \leq s$, such that for every $f_0 \in C^2(\torus, \R^d)$  the set of parameters $\bt \in \R^s$ for which $f_\bt$ is in $\mathcal{T}$
has zero Lebesgue measure.
\end{proposition}
	
\section{Absolute continuity of the SRB measure}
	
\subsection{The semi-norm}
	
In order to obtain the absolute continuity of the measure $\mu_T$, we will use a semi-norm that measures the regularity of the measure, in a similar way to the $L^2$-norm of the density function. In this Section, we will state some tools that will be used to prove Theorem \ref{teo.d}.
	
Let us denote by $m_d$ the usual Lebesgue measure on $\R^d$, $m$ the normalized Lebesgue measure on $\torus$ and $\pi_1:\torus\times\R^d\to\torus$ the projection into the first coordinate.
	
\begin{definition}
Given finite measures $\mu_1$ and $ \mu_2$ on $\R^d$ and $r>0$, we define the bilinear form
\begin{equation}
\langle\mu_1,\mu_2\rangle_r=\underset{\R^d}{\int}\mu_1(B(x,r)) \mu_2(B(x,r))\,dm_d(x)
\end{equation}
where $B(z,r)\subset \R^d$ is the ball of radius $r$ centered in $z$. The \textbf{norm} $\|\mu\|_r$ of the finite measure $\mu$ is defined by
$$\|\mu\|_r=\sqrt{\langle\mu,\mu\rangle_r}.$$
\end{definition}
	
	The value of $\|\mu\|_r$ when $r$ tends to $0$ is related to the $L^2$-norm of the density of $\mu$, as the following lemma states.
	
\begin{lemma}\label{l.density1}
There exists a constant $C_d$ such that if a finite measure $\mu$ on $\R^d$ satisfies
	$$\liminf_{r\to 0^+}\frac{\|\mu\|_r}{r^d}<\infty ,$$
	then $\mu$ is
	absolutely continuous with respect to the Lebesgue measure on $\R^d$, its density  $\frac{d\mu}{d m_d}$ is in $L^2(\R^d)$ and satisfies $\|\frac{d\mu}{d m_d}\|_{L^2(\R^d)} \le C_d \liminf_{r\to 0^+}\frac{\|\mu\|_r}{r^d}$.
	\end{lemma}
	
\begin{proof}
Let  $C_d$ be the constant such that $m_d(B(z,r)) = C_d^{-1} r^d$ for every $z\in\R^d$ and every $r>0$.
		
Define the function $J_r(z)=\frac{\mu(B(z,r))}{m_d(B(z,r))}$ and note that $  \|J_r\|_{L^2(\R^d)} =C_d  \frac{\|\mu\|_r}{ r^{d}} $. Then the condition $\liminf_{r\to 0^+}\frac{\|\mu\|_r}{r^d}<\infty$ imply that there exists a uniformly bounded subsequence of $J_{r}$ in $L^2(\R^d)$ with $r\rightarrow 0^+$.
Thus we can consider a subsequence $r_n$ such that $J_{r_n}$ converges weakly to some $J_\infty\in L^{2}$,
following that
		$$
		\int_{\R^d}\phi J_\infty\,dm_d=\lim_{n\to \infty}\int_{\R^d}\phi J_{r_n}\,dm_d=\int_{\R^d}\phi \,d\mu
		$$
		for every continuous function $\phi$ with compact support. So we have $\frac{d\mu}{d m_d}(z)= J_\infty (z)$ and
		$$\|\frac{d\mu}{d m_d}\|_{L^2(\R^d)} \le C_d \liminf_{r\to 0^+}\frac{\|\mu\|_r}{r^d}$$
	\end{proof}
	
	Given any finite measure $\mu$ on $\torus\times \R^d$, we consider $\{\mu_x\}_{x\in\torus}$  the disintegration of $\mu$ with respect to the partition of $\torus \times \R^d$ into $\{x\}\times\R^d$, $x\in\torus$. We will define a semi-norm for measures on $\torus\times \R^d$ integrating the norm $\|\mu_x\|_r$ into the torus $\torus$.
	
	\begin{definition}
		Given a finite measure $\mu$ on $\torus\times\R^d$ and $r>0$, we define the \textbf{semi-norm} $|||\mu|||_r$ by
		\begin{equation}
		|||\mu|||^2_r=\underset{\torus}{\int} \frac{\|\mu_x\|^2_r}{r^{2d}}\,dm(x)
		\end{equation}
	\end{definition}
	
	As a consequence of Lemma \ref{l.density1} we have a criterion of absolute continuity for measures $\mu$ on $\torus\times\R^d$, provided by the following:
	
	\begin{corollary}\label{c.density1}
		There exists a constant $C_d$ such that if a  finite measure $\mu$ on $\torus\times \R^d$ satisfies $(\pi_1)_*\mu=m$ and
		$$ \liminf_{r\to 0^+} |||\mu|||_r < \infty, $$
		then $\mu$ is absolutely continuous with respect to the volume $v=m\times m_d$ on $\torus\times\R^d$,
		its density $\frac{d\mu}{dv}$ is in $L^2(\torus\times\R^d)$
		and satisfies $\|\frac{d\mu}{dv}\|_{L^2(\torus\times\R^d)} \leq C_d  \liminf_{r\to 0^+} |||\mu|||_r$.
	\end{corollary}
	\begin{proof}
		From Fatou's inequality, we have
		$$
		\int_{\torus}\liminf_{r\to 0} \frac{\|\mu_x\|_r^2}{r^{2d}}\,dm(x)\le \liminf_{r\to 0}\int_{\torus}\frac{\|\mu_x\|_r^2}{r^{2d}}\,dm(x) <\infty.
		$$
	This inequality and the previous lemma imply that, for Lebesgue almost every $x\in \torus$, $\mu_x$ is absolutely continuous
		with respect to the Lebesgue measure on $\R^d$ and the density function $g_x$ satisfies $\|g_x\|_{L^2(\R^d)}\le C_d\liminf_{r\to 0}\frac{\|\mu_x\|_r}{r^{d}}$.
		
		Defining  $g:\torus\times\R^d\to \R$ by $g(x,y)=g_x(y)$, we have
		$g\in L^{2}(\torus\times\R^d)$ with $\|g\|^2_{L^{2}(\R^u \times \R^d)}\leq C_d^2 \liminf_{r\rightarrow 0^+}  \int_{\torus}\frac{\|\mu_x\|_r^2}{r^{2d}}\,dm(x)$ and
		\begin{align*}
		\int_{\torus\times\R^d}\phi(x,y)g(x,y)\,dm(x)dm_d(y)&=\int_\torus\left[\int_{\R^d}\phi(x,y)g_x(y)dm_d(y)\right]dm(x)\\
		&=\int_\torus\int_{\R^d}\phi(x,y)\,d\mu_x(y)dm(x)\\
		&=\int_{\torus\times\R^d}\phi(x,y)\,d\mu(x,y)
		\end{align*}
		for every continuous function $\phi$ with compact support on $\torus\times\R^d$. Above we used that $\hat{\mu}=(\pi_1)_*\mu=m$ in the disintegration of $\mu$ with respect to the partition into sets $\{x\}\times\R^d$. Then, $\mu$ is absolutely continuous with respect to the volume measure on $\torus\times\R^d$ and $g=\frac{d\mu}{dv}$.
	\end{proof}
	
	\subsection{Useful lemmas on the transversality}
	
	In the following we will state a quantitative lemma of change of variables by maps $g$ satisfying $\fm(Dg(x)) \geq \delta$.
	
	\begin{lemma}\label{c.submersoes}
		Given an open convex set $U\subset\R^u$, $g\in C^2(U,\R^d)$ with $\|g\|_{C^2}\leq 2\alpha_0$ and a real number $ \delta>0$ such that
		$\diam(U)\le \delta/4\alpha_0$ and
		$\fm(Dg(x)) \ge\delta$ for every $x\in U$.
		Then there exists a constant $C_{\delta}>0$  such that
		\begin{equation}
		m_u(g^{-1}B(z,r))< C_{\delta} r^d
		\end{equation}
		for every $z\in\R^d$ and every $r>0$.
	\end{lemma}
	
	\begin{proof}
		Let us first suppose that $u=d$. Using the convexity of $U$ and the Taylor Formula, for $x\neq y\in U$ we have
		$
		g(y)=g(x)+Dg(x)(y-x) + \theta(y-x),
		$
		where $\|\theta(y-x)\|\le \ds \alpha_0\|y-x\|^2$. This implies that
		$$\|g(y)-g(x)\|\ge \delta\|y-x\|-2\alpha_0\|y-x\|^2=\Big(\delta/2- \alpha_0\|y-x\|\Big)\|y-x\|>0.$$
		
		Hence $g$ is injective, which implies that it is a diffeomorphism into the image. By the Change of Variables Theorem, for every $z\in\R^d$ and $r>0$ we have:
		\begin{align*}
		m_d(g^{-1}(B(z,r))) 
		&= \int_{g^{-1}(B(z,r)\cap g(U))} 1 \, dm_d(x)\\
		&= \int_{B(z,r)\cap g(U)}|\det{Dg^{-1}(y)}|\, dm_d(y)\\
		&\le \delta^{-d} C_d^{-1}r^d
		\end{align*}
		
		Now, suppose that $u>d$. Fix a point $a_0\in U$, then there is a $d$-dimensional subspace $W\subset\R^u$ such that $\fm(Dg(a_0)_{|W})\ge \delta$. As $\diam(U)<\delta/{4\alpha_0}$ and $\|g\|_{C^2}\leq 2\alpha_0$, we have $\fm(Dg(x)_{|W}) \geq \delta/2$ for every $x\in U$. 
		
		For every $a\in\R^u$, we denote $W_a:=a+W=\{a+x\in\R^u: x\in W\}$ and consider the usual orthogonal projection $\pi_{W^{\perp}}:\R^u \to W^\perp$, for each $y\in \pi_{W^{\perp}}(U)$ define $U_y=W_y\cap U$ and $h_y:=g_{|U_y}$. So, $h_y$ is a $C^2$ transformation such that $\|h_y\|_{C^2} \leq 2\alpha_0$ and $\fm(Dh_y(x))\ge \delta/2$ for every $x\in U_y$. Therefore  $h_y$ is a diffeomosphism over the image $h_y(U_y)$ and
		$$
		m_d(h_y^{-1}(B(z,r)))\le \delta^{-d} C_d^{-1} r^d
		$$
		for every $z\in\R^d$.
		
Finally, we use that $\diam(\pi(U))<\delta/4\alpha_0$ and Fubini's Theorem:
		\begin{align*}
		m_u(g^{-1}(B(z, r)))&=\int_{\pi_{W^{\perp}}(U)}m_d(h_y^{-1}(B(z,r)))\,dm_{u-d}(y)\\
		&\le m_{u-d}(\pi_{W^{\perp}}(U))\delta^{-d} C_d^{-1} r^d \\
		&\le C_{u-d}^{-1} \Big( \frac{\delta}{8\alpha_0} \Big)^{u-d} \delta^{-d} C_d^{-1} r^d = C_{\delta} r^d
		\end{align*}
		for every $z\in\R^d$ and $r>0$.
	\end{proof}
	
	\begin{lemma} Given $\bc$ in $I^p$ and  $\ba, \bb \in I^q(\bc)$, if  $\ba$ and $\bb$ are transversal on
		$\bc$, then
		$$\fm (DS_\bc(x,\ba\bu) - DS_\bc(x,\bb\bv) ) > \theta^{q}\alpha_{0} $$
		for every $x\in \overline{\mathcal{R}(\bc)}$ and every $\bu \in I^{\infty} (\ba)$ and $\bv \in I^{\infty} (\bb)$.
	\end{lemma}
	\begin{proof}
		
		For every unitary vector $v\in\R^u$, we have
		$$
		\big|\big| (  DS_\bc(x,\ba\bu) - DS_\bc(x,\ba) ) v \big|\big| = \big|\big|    \sum_{i=q+1}^{\infty} \sC^{i-1} Df\big([\ba\bu]_{i} (x)\big)E^{-i}  v \big|\big| \leq  \|\sC\|^{q}\|E^{-1}\|^{q}  \alpha_0
		$$
		
		Analogously, we also have $\| (  DS_\bc(x,\bb\bv) - DS_\bc(x,\bb) ) v  \| \leq  \|\sC\|^{q}\|E^{-1}\|^{q}  \alpha_0 $. By assumption, there is a $d$-dimensional subspace $W\subset \R^d$ such that $\|(DS_\bc(x,\ba) - DS_\bc(x,\bb) ) w > 3\theta^q \alpha_0 \|w\|$ for every $w\in W$.
		
		Reminding that $\theta=\|C\| \|E^{-1}\|$, for every vector $w\in W$ we have
		\begin{align*}
		\| (  DS_\bc(x,\ba\bu) - &DS_\bc(x,\bb\bv) ) w\| \geq  \| (  DS_\bc(x,\ba) - DS_\bc(x,\bb) )  w \|   -\\
		&- \| (  DS_\bc(x,\ba\bu) - DS_\bc(x,\ba) ) w\|   -   \| (  DS_\bc(x,\bb\bv) - DS_\bc(x,\bb) ) w\|  \\
		&> 3\theta^{q}  \alpha_0 \|w\| - \theta^{q}  \alpha_0 \|w\|  - \theta^{q}  \alpha_0 \|w\|  = \theta^{q}  \alpha_0 \|w\|.
		\end{align*}
It follows what we want taking the infimum on $w$.
\end{proof}
	
	\subsection{Symbolic description of $\mu_T$}
	
	Let us give a symbolic description of the dynamic $T$ and of the measure $\mu_T$. Consider $\hat{M} = \bigsqcup_{i\in \overline{I}} \mathcal{R}(i) \times I^\infty(\mathcal{R}(i))\subset \torus\times I^\infty$ and
	$
	\ds\hat{T}: \hat{M} \to \hat{M}$ given by
	\begin{equation}
	\hat{T}(x,\ba)=(\sE(x),\pi(x)\ba)
	\end{equation}
	
	$\hat T$ is well defined, because if $\ba\in I^{\infty}(x)$ then $\pi(x)\ba\in I^{\infty}(\sE(x))$. Moreover,
	for $\ba\in I^{\infty}(x)$, we have the relation
	\begin{equation}
	S(\sE(x),\pi(x)\ba)=f(x) + \sC(S(x,\ba)).
	\end{equation}
	
	This implies that $\hat{T}$ is semi-conjugated to $T$, that is, $T\circ h=h\circ \hat T$, where the semi-conjugation
	$h:\ds \hat{M} \to \torus \times \R^d$ is given by
	\begin{equation}
	h(x,\ba)=(x,S(x,\ba))
	\end{equation}
	
	Considering the numbers
	$P_{ij}=\frac{1}{N}$ if $\mathcal{R}(i)\subset E(\mathcal{R}(j))$ and $P_{ij}=0$ otherwise. The  probability $\hat{\mu}$ on $\torus \times I^{\infty}$ is defined for any $U\subset \mathcal{R}(i)$ and any cylinder
	$V=[1;a_1,\dots,a_n]$ as
	\begin{equation}
	\hat{\mu}(U\times V):=m(U)P_{ia_1}P_{a_1a_2}\dots P_{a_{n-1}a_n}
	\end{equation}
	and extended to the measurable subsets of $\torus \times I^{\infty}$. Note that $\hat{\mu}(\hat{M})=1$. We will check that $h_*\hat{\mu}$ is the SRB measure of $T$.
	
	\begin{lemma}
		The measure $h_*\hat{\mu}$ is the SRB measure of $T$.
	\end{lemma}
	\begin{proof}
		For every integers $n\geq 1$ and $k\geq 0$ we have:
		$$\hat{T}^{-n-k}(U \times [1;a_1,\cdots,a_n]) =
		E^{-k} \big( E^{-n}(U)\cap E^{-n+1}(\mathcal{R}(a_1)) \cap \cdots \cap  \mathcal{R}(a_n) \big) \times I^\infty $$
		
		It follows that $\hat{\mu}$ is an invariant probability for $\hat{T}$, because for any $U\subset \mathcal{R}(i)$
		\begin{align*}
		\hat{\mu}(\hat{T}^{-1}(U\times [1;a_1,\dots,a_n])) &=\hat{\mu}(\left(\mathcal{R}(a_1)\cap \sE^{-1}(U)\right)\times [1;a_2,a_3,\dots,a_n]) \\
		&=P_{ia_1}m(U)P_{a_1a_2}\dots P_{a_{n-1}a_n} 
		=\hat{\mu}(U\times [1;a_1,\dots,a_n]).
		\end{align*}
		
		To see that $\hat{\mu}$ is mixing for $\hat{T}$, for any sets $A_1, A_2$ of the form $U\times [1;a_1,\dots,a_n]$ with $U\subset\mathcal{R}(j)$ we consider an integer $N\geq 1$ such that $\hat{T}^{-N}(A_i) = \tilde{U}_{A_i} \times I^\infty$ for $i=1,2$. Since $m$ is mixing for $E$ and $\hat{\mu}(\tilde{U}_{A_i})= \hat{\mu}(\hat{T}^{-N}(\tilde{U}_{A_i}))  = \hat{\mu}(\tilde{U}_{A_i} \times I^\infty) = m (\tilde{U}_{A_i})$, we have:
		\begin{align*}
		\hat{\mu}(\hat{T}^{-k}A_1 \cap A_2) &=
		\hat{\mu}(\hat{T}^{-N-k}(A_1) \cap \hat{T}^{-N}(A_2)) = \hat{\mu}\big( (E^{-k}\tilde{U}_{A_1} \times I^\infty)  \cap (\tilde{U}_{A_2} \times I^\infty ) \big)   \\
		&=m(E^{-k}\tilde{U}_{A_1}\cap \tilde{U}_{A_2})
		\to m(\tilde{U}_{A_1}) m(\tilde{U}_{A_2})
		= \hat{\mu}(A_1) \hat{\mu}(A_2).
		\end{align*}
		
		For measurable subsets $A_1, A_2$, one proves that $\hat{\mu}(\hat{T}^{-k}A_1 \cap A_2)  \rightarrow \hat{\mu}(A_1) \hat{\mu}(A_2)$ from standard arguments approximating of $A_1, A_2$ by finite unions of sets as above.
		
		So the measure $h_*\hat{\mu}$ is invariant and ergodic with respect to $T$.
		Since $(\pi_1)_*h_*\hat{\mu} = m$ we get that $m(\pi_1(B(h_*\hat{\mu})) )=1$, where  $B(h_*\hat{\mu}) $ is the basin of $h_*\hat{\mu}$. Finally, note that if $(x,y)\in B(h_*\hat{\mu})$ then the whole set $\{x\} \times \R^d$ is in $B(h_*\hat{\mu})$ because the vertical fibers are stable manifods for $T$. Then $B(h_*\hat{\mu})$ has full volume in $\torus \times \R^d$, implying that $h_*\hat{\mu}$ is the SRB measure of $T$.
	\end{proof}
	
	\subsection{The Main Inequality}
	
	The transversality between $\ba$ and $\bb$ on $\bc$ gives an estimate on the integral of  $ { \langle T^q_* \mu_{\ba(x)},T^q_* \mu_{\bb(x)}\rangle_r }$ over ${P(\bc)}$.
	
	\begin{proposition}\label{p.tranversality}
		Let $\bc\in I^p$ and $\ba,\bb\in I^q(\bc)$. If $\ba$ and $\bb$ are transversal on
		$\bc$, then there exists a constant $C_1>0$, depending on $q$, such that
		\begin{equation}
		\int_{\mathcal{R}(\bc)}{ \langle T^q_* \mu_{\ba(x)},T^q_* \mu_{\bb(x)}\rangle_r }\,dm(x)\le C_1 {r^{2d}}
		\end{equation}
		for all $r>0$.
	\end{proposition}
	
	\begin{proof}
		Denoting the indicator function of the ball $B(0,r)\subset \R^d$ by $\mathbbm{1}_r$, and define $\mathcal{I} :=\int_{\mathcal{R}(\bc)}{ \langle T^q_* \mu_{\ba(x)},T^q_* \mu_{\bb(x)}\rangle_r }\,dm(x)$.  For $x \in \mathcal{R}(\bc)$, we have:
		\begin{align*}
		\langle T^q_*(\mu_{\ba(x)})&,T^q_*(\mu_{\bb(x)})\rangle_r =\int_{\R^d} T^q_*(\mu_{\ba(x)})(B(z,r))  T^q_*(\mu_{\bb(x)})(B(z,r))\,dm_d(z)\\
		&=\int_{\R^d} \int_{\R^d}\int_{\R^d} \mathbbm{1}_r(z_1-z_2)\mathbbm{1}_r(t-z)\,d(T^q_*\mu_{\ba(x)}(z_1)\times T^q_*\mu_{\bb(x)}(z_2) ) \,dm_d(z)\\
		&\le C_d  r^d\int_{\R^d}\int_{\R^d} \mathbbm{1}_{2r}(z_1-z_2)\,d(T^q_*\mu_{\ba(x)}\times T^q_*\mu_{\bb(x)}) (z_1,z_2)
		\end{align*}
		
		A symbolic description for $\mu_x$ is given by $\mu_x=(h_x)_*(\hat{\mu}_x)$, where $h_x(\ba)=S(x,\ba)$ and
		$\hat{\mu}_x([1;a_1,\dots,a_n]) = P_{\pi(x)a_1}P_{a_1a_2}\dots P_{a_{n-1}a_n}$.  In particular, $\mu_x$ is constant in each $\mathcal{R}(i)$.
		Considering $\Psi_\ba: I^{\infty}(\ba(x)) \to \{x\}\times\R^d$ given by $\Psi_\ba(\bu)=(x,S(x,\ba\bu))$, it is valid that $\Psi_\ba(\bu)=(T^{q}\circ h)(\ba(x),\bu)$ for all $\bu \in I^{\infty}(\ba(x))$, and so that $(\Psi_\ba)_*\hat{\mu}_{\ba(x)}=T^{q}_*(h_{\ba(x)})_*\hat{\mu}_{\ba(x)}=T^{q}_*\mu_{\ba(x)}$. Analogously, we have that $(\Psi_\bb)_*\hat{\mu}_{\bb(x)}=T^{q}_*\mu_{\bb(x)}$, where $\Psi_\bb(\bb(x),\bv)=(x,S(x,\bb\bv))$.
		
		Take $(z_1,z_2)=(\Psi_\ba,\Psi_\bb)(\bu,\bv)$
		and denote $I_{\ba,\bb}(x) =  I^\infty(\ba(x))\times I^\infty(\bb(x))$
		, then
		$ \langle T^q_*(\mu_{\ba(x)}),T^q_*(\mu_{\bb(x)})\rangle_r$ is at most
		$$
		C_d r^d\int_{I_{\ba,\bb}(x)}  \mathbbm{1}_{2r}(S_c(x,\ba\bu)-S_c(x,\bb\bv))\,d(\hat{\mu}_{\ba(x)}\times\hat{\mu}_{\bb(x)})(\bu,\bv)
		$$
		
		Integrating on $\mathcal{R}(\bc)$, we have:
		\begin{equation}\label{integral.measure}
		\mathcal{I} \le {C_d}{r^{d}} 
		{\int}_{I_{\ba,\bb}(x)} m(x\in \mathcal{R}(\bc) , \|S(x,\ba\bu)-S(x,\bb\bv)\|<2r)\,d\hat{\mu}_{\ba(x)} d\hat{\mu}_{\bb(x)}(\bu,\bv) 
		\end{equation}

		To finish, it is enough to prove the following claim.
		\begin{claim}\label{eq.local1} There exists a constant $C_2>0$, depending on $q$, such that
			$$  m(x\in \mathcal{R}(\bc):\|S(x,\ba\bu)-S(x,\bb\bv)\|<2r)\le C_2 r^d$$
			for every $\bu$ and $\bv$ such that $\ba\bu\in I^\infty(\bc)$ and $\bb\bv\in I^\infty(\bc)$.
		\end{claim}
		\begin{proof}[Proof of Claim \ref{eq.local1}]
			
			For every $x\in \mathcal{R}(\bc)$, $\ba\in I^{n}(\bc)$, $1\leq n \leq \infty$ and $1\leq j \leq n$ are well defined the maps
			$h_{j,x}^{\ba}:=g_{[a]_{j}(x),[a]_{j-1}(x)} \circ\dots g_{[a]_{1}(x),x}:B(x,\gamma)\to \torus$
			and
			\begin{equation}\label{eq.S_chapeu}
			\widehat{S}(\cdot,\ba(x))=\sum_{j=1}^{n}\sC^{j-1}\circ f\circ h_{j,x}^{\ba}:B(x,\gamma)\to \R^d,
			\end{equation}
			
			Fix some $x_0 \in \mathcal{R}(\bc)$ and define
			$g: B(x_0,\gamma)\supset \mathcal{R}(\bc) \to \R^d$ by
			
			\begin{equation}
			g(x):=\widehat{S}(x,\ba\bu(x_0))-\widehat{S}(x,\bb\bv(x_0)).
			\end{equation}
			
			Since $S(\cdot,\ba\bu)_{|\mathcal{R}(\bc)}$ and $S(\cdot,\bb\bv)_{|\mathcal{R}(\bc)}$ coincide with the restrictions of $\widehat{S}(\cdot,\ba\bu(x_0))$ and $\widehat{S}(\cdot,\bb\bv(x_0))$ 
			to $\mathcal{R}(\bc)$, we have that $\fm(Dg(x)) > \theta^q\alpha_0$ for every $x \in \mathcal{R}(\bc)$. We also have that $\|g\|_{C^{2}} \leq 2\alpha_0$. Defining $\delta=\theta^q\alpha_0$ and considering a covering of $\mathcal{R}(\bc)$ with at most $M=\lfloor (8\alpha_0/\delta)^u\rfloor$ balls centered in points in $\mathcal{R}(\bc)$ with radius at most $\delta/4\alpha_0$, for each of such balls $U$ we apply Lemma \ref{c.submersoes} to $g_{|_{U}}$, obtaining that $m_u((g_{|U})^{-1} B(0,2r) ) < C_\delta r^d$. Consequently, it follows that
			$$m(\{x\in\mathcal{R}(\bc) , \|S(x,\ba\bu)-S(x,\bb\bv) \| < 2r \} ) \leq M C_\delta r^d = C_2 r^d$$
		\end{proof}
		
		Putting together Claim \ref{eq.local1} and (\ref{integral.measure}), we have that
		$$ 
		\mathcal{I}
		\leq {C_d}{r^d} C_2 r^d= C_1 r^{2d}$$
	\end{proof}
	
The main step to prove Theorem \ref{teo.a} is given by the following inequality.
	
	\begin{proposition}[Main Inequality]\label{main.ineq}
		For every integer $1\leq q <\infty$, it is valid
		\begin{equation}
		|||\mu|||_r^2  \leq \frac{\tau(q)}{J^{q}} |||\mu|||_{\underline{\lambda}^{-q}r}^2 + C_1
		\end{equation}
		for all $r>0$.
	\end{proposition}
	
	\begin{proof}
		Given $q$, fix
		an integer $p_0$ large enough such that
		$$\tau(q)= \underset{\bc\in I^{p_0}}{\max} \underset{\ba\in I^{q}(\bc)}{\max} \# \{ \bb\in I^{q}(\bc) | (\ba,\bb) \in \Gamma(q,\bc) \}.$$
		
		Define $\Gamma(q)$ the set of triples $(\ba,\bb,\bc)$ with $\bc \in I^{p_0}$ and $(\ba,\bb) \in I^{q}(\bc)\times I^{q}(\bc)$ such that $\ba$ and $\bb$ are not transversal on $\bc$. The invariance of $\mu$ with respect to $T$ and the uniqueness of the disintegration gives the relation
		$$
		\mu_x=N^{-q}\sum_{\ba\in I^{q}(\bc)}T^q_* \mu_{\ba(x)} 
		$$
		
		Therefore we may decompose $|||\mu|||^2_r $ into the following sum:
		
		\begin{align*}
		|||\mu|||^2_r  &= { N^{-2q} \sum_{\bc\in I^{p_0}} \sum_{(\ba,\bb)\in I^q(\bc) \times I^q(\bc)}
			\int_{\mathcal{R}(\bc)} \frac{ \langle T^q_* \mu_{\ba(x)},T^q_* \mu_{\bb(x)}\rangle_r }{r^{2d}}\,dm(x)}\\
		&= N^{-2q} \sum_{(\ba,\bb,\bc)\in \Gamma(q)}\int_{\mathcal{R}(\bc)} \frac{ \langle T^q_* \mu_{\ba(x)},T^q_* \mu_{\bb(x)}\rangle_r }{r^{2d}}\,dm(x) \\
		&\quad \quad\quad \quad\quad \quad +  N^{-2q} \sum_{(\ba,\bb,\bc)\notin \Gamma(q)}\int_{\mathcal{R}(\bc)} \frac{ \langle T^q_* \mu_{\ba(x)},T^q_* \mu_{\bb(x)}\rangle_r }{r^{2d}}\,dm(x)
		\end{align*}
		
		Let us denote by $\mathcal{S}_1$ the first sum above and $\mathcal{S}_2$ the second sum.
		The transversality condition is used to bound $\mathcal{S}_2$ from above: Proposition~\ref{p.tranversality} implies that  $\int_{\mathcal{R}(\bc)}{ \langle T^q_* \mu_{\ba(x)},T^q_* \mu_{\bb(x)}\rangle_r }\,dm(x)\le C_1 {r^{2d}}$, then it follows that
		$$\mathcal{S}_2  \leq N^{2q}  N^{-2q}C_1 = C_1.$$
		
		In order to bound from above the first sum, we consider $H = |\det \sC|$ and that note that the restriction of $T^{q}$ to each vertical fiber is an affine contraction whose smaller contraction rate is $\underline{\lambda}=\fm(\sC)$, which implies that$$\|T^{q}_*(\mu_{\ba(x)})\|^{2} _{r} \le H^q \|\mu_{\ba(x)}\|^{2}_{\underline{\lambda}^{-q}r}.$$ 
		
		For $\ba, \bb \in I^q(\bc)$, it holds:
		\begin{align*}
		\langle T^{q}_* \mu_{\ba(x)}, T^{q}_* \mu_{\bb(x)} \rangle_{r} \leq
		\|T^{q}_* \mu_{\ba(x)} \|_{r} \|T^{q}_* \mu_{\bb(x)} \|_{r}
		\leq    H^q \frac{ \|\mu_{\ba(x)}\|_{\underline{\lambda}^{-q}r}^{2} +   \|\mu_{\bb(x)}\|_{\underline{\lambda}^{-q}r}^{2}  }{2}
		\end{align*}
		
		Making a change of variables, we have
		\begin{align*}
		\sum_{\ba\in \overline{I}^{q} } \int_{\torus} \frac{\|\mu_{\ba(x)}\|_{\underline{\lambda}^{-q}r}^{2}}{r^{2d}} dm(x) = 
		\sum_{\ba\in \overline{I}^{q} } \int_{\mathcal{R}(\ba)} N^{q} \frac{\|\mu_{y}\|_{\underline{\lambda}^{-q}r}^{2}}{r^{2d}} dy =
		N^q \underline{\lambda}^{2dq} |||\mu|||_{\underline{\lambda}^{-q}r}^2
		\end{align*}
		
		Then:
		\begin{align*}
		\mathcal{S}_1   
\leq r^{-2d} N^{-2q}  H^q \tau(q)  \sum_{\ba\in \overline{I}^{q} } \int_{\torus} \|\mu_{\ba(x)}\|_{\underline{\lambda}^{-q}r}^{2} dm(x)
	\leq \frac{\tau(q)}{J^q} |||\mu|||_{\underline{\lambda}^{-q}r}^2
		\end{align*}
		
		And it follows what we want.
	\end{proof}
	
	\subsection{Proof of Theorems \ref{teo.a} and \ref{teo.d}:}
	
	Finally, we are able to prove Theorems \ref{teo.a} and \ref{teo.d}.
	
\begin{proof}[Proof of Theorem \ref{teo.d}]
		By the transversality condition of $T$, there exist an integer $q_0$ such that $\tau(q_0) < {J^{q_0}}$. Consider $\rho= \frac{\tau(q_0)}{J^{q_0}} < 1 $ and take $C_1$ as in Proposition \ref{main.ineq} for this $q_0$. For $r>0$, we have
		$$|||\mu|||_r^2  \leq \rho^{q_0} |||\mu|||^2_{\underline{\lambda}^{-q_0}r} + C_1$$
		Fixing some $r_0>0$, it follows that
		\begin{align*}
		|||\mu|||^2_{\underline{\lambda}^{nq_0} r_0} &\leq \rho |||\mu|||^2_{\underline{\lambda}^{(n-1)q_0}r_0} + C_1 \\
		&\leq \rho^{2} |||\mu|||^2_{\underline{\lambda}^{(n-2)q_0}r_0} + C_1 (\rho  + 1) \leq  \dots  \\
		&\leq \rho^{n} |||\mu|||^2_{r_0}  + C_1 (\rho^{n-1} + \cdots + \rho  + 1)  \leq |||\mu|||^2_{r_0} +  \frac{C_1}{1-\rho}
		\end{align*}
		
		So we have $\liminf_{r\rightarrow 0^+} |||\mu|||_r  < \infty$, thus Lemma 3.1 implies that the SRB measure is absolutely continuous and its density is in $L^2(\torus\times\R^d)$.
		
		The openness follows from an argument of continuity. The transversality of a pair of functions $S_\bc(\cdot,\ba)$ and $S_\bc(\cdot,\bb)$ is open in $(\sC,f)$ because  $S_\bc(\cdot,\ba)$ is continuous on $\ba$, on $\sC$ and on $f\in C^{2}(\torus,\R)$. This implies that, for fixed $q$, $\tau(q)$ is upper semi-continuous and that the condition $\tau(q)<|\det \sC \det \sE|^{q}$ is open on $f$.
	\end{proof}
	
	To prove Proposition \ref{non.empty} we will use Proposition~\ref{zero.measure}, which will be proved in the Section 4.
	
	\begin{proof}[Proof of Proposition \ref{non.empty}]
		Consider the maps $\sE_i:\R^{u_i}\to\R^{u_i}$ by $\sE_i(x_i) = \mu_i x_i$ and  $\sC_i:\R\to\R$ by $\sC_i(y)=\lambda_i y$, for $i=1,\cdots, d$. Note that $\sC_i \in C(1,E_i)$, for every $i=1,\cdots,d$, then Proposition~\ref{zero.measure} implies that there exists $f_i \in  C^2(\torus,\R^d)$ such that the map $T_i:\T^{u_i}\times\R \to \T^{u_i}\times\R$ given by $T_i(x,y)=(\sE_ix,\lambda_iy+f_i(x))$
		satisfies  ${\limsup}_{q\to\infty} \frac{\log \tau(q)}{q}<\lambda_i \mu_i^{u_i}=: J_i$. This condition on the $\limsup$ implies that there exists an integer $q_i$ such that $\tau(q)<J_i^q$ for every $q\geq q_i$.
		
		Now, we consider $f\in C^2(\torus,\R^d)$ by $f(x_1,\cdots,x_d)=(f_1(x_1),\cdots, f_d(x_d))$, $x_i\in \mathbb{T}^{u_i}$. The dynamic $T=T(E,C,f)$ has the form
		$$
		T(x_1,\dots,x_d,y_1,\dots,y_d)=(\mu_1x_1,\dots,\mu_dx_d,\lambda_1y_1+f_1(x_1),\dots,\lambda_dy_d+f_d(x_d))
		$$
		where $x_i\in\T^{u_i}$ and $y_i\in\R$.
		
		We claim that $T$ satisfies the transversality condition. To verify it, we consider the product alphabet $I_T=I_{T_1}\times \dots \times I_{T_d}$,
		the product partition $\mathcal{R}_{T}^n=\mathcal{R}_{T_1}^n\times\dots\times \mathcal{R}_{T_d}^n$ and notice that if $\ba=(\ba_1,\dots,\ba_d)$ is in $I^n=I_T^{n}$ then  $S(\cdot;\ba)$ is the product $S(x_1,\dots,x_d;\ba_1,\dots,\ba_d)=(S(x_1;\ba_1),\dots,S(x_d;\ba_d))$, with $\ba_i\in I_{T_i}^n$. For $q \geq \max_i{q_i}$, we have $\tau(q) \leq J_1^q \cdots J_u^q = (|\det \sC  \det \sE|)^q$, what means that $T$ satisfies the transversality condition.
	\end{proof}
	
	\begin{proof}[Proof of Theorem A]
		Consider the set $\mathcal{U}$ formed by the pairs $(\sC,f) \in C(d)\times C^2(\torus,\R^d)$ such that the dynamic $T=T(\sE,\sC,f)$ satisfies the transversality condition.    Proposition \ref{non.empty} implies that $\mathcal{U}$ is nonempty. Theorem \ref{teo.d} implies that $\mathcal{U}$ is open and the absolute continuity of $\mu_T$ for every $(\sC,f)\in \mathcal{U}$.
	\end{proof}
	
	\section{Genericity of the Transversality Condition}
	
	The construction of generic families in Proposition \ref{generic.families} can be reduced to a local version of itself.
	
	\begin{proposition}\label{local.version}[Local version of Proposition \ref{generic.families}]
		Given $u\geq d$, $\sE\in E(u)$, $\sC\in C(d;E)$, for every point $x\in\torus$
		there exists an integers $n_0$ such that
		for every $n \geq n_0$
		there exists a neighborhood $U_x$ of $x$ and functions $\phi_k \in C^{\infty}(\torus,\R^d)$, $1\leq k \leq s$, such that  for every $f_0 \in C^2(\torus, \R^d)$  the corresponding family $T_{\bt}^{n}$ is $n$-generic on $U_x$.
	\end{proposition}
	
	Assuming Proposition \ref{local.version} we can prove Proposition \ref{generic.families}.
	
	\begin{proof}[Proof of Proposition \ref{generic.families}]
		Consider $n_0$ as given by Proposition 6, for every $n \geq n_0$ there exists a finite set of points $x_{k}$, $1\leq k \leq m$, such that the neighborhoods $U_{x_{k}}$ given by Proposition \ref{local.version} cover $\torus$. Associated to each $x_k$ and $U_{x_k}$ we have the $C^{\infty}$ functions
		$\phi^k_{1},\cdots,\phi^k_{m(k)}$.
		
		Denoting by $\{ \phi_i \}_{i=1,\cdots, s}$ the union of these functions and $t_i$ the parameter corresponding to the function $\phi_i$. For every $x\in\torus$, take $k_0$ such that $x \in U_{x_{k_0}}$, the Jacobian of $\fpsi_{x,\sigma}(\bt)$ is greater than the Jacobian of $\fpsi_{x,\sigma}(\bt)$ restricted to the subspace $W$ generated by $\{ t^{k_0}_i\}_{i=1,\cdots,m(k_0)}$, which satisfies $\Jac\big((\fpsi_{x,\sigma})(\bt)|_{W}\big)>\frac{1}{2}$.
	\end{proof}
	
	\subsection{Construction of generic families}
	
	\begin{proof}[Proof of Proposition \ref{local.version}]
		
		Given $x\in\torus$ and $n\in\N$, we will consider a family of functions $\phi_k^\ba$, $1\leq k \leq ud$, $\ba \in I^n(x)$, that are supported on a neighborhood of the pre-image $\ba(x)$. Deformations of $f$ along $\phi_k^{\ba}$ will be few relevant in the expression of $S(x,\bb;\bt)$ for $\bb\in I^\infty(x)$ if $[\bb]_i(x)$ is far from $\ba(x)$ for small values of $i$.
		
		For $\bb\in I^n(x)$, define $\mathcal{E}(\bb,x)=\{ \ba \in I^n(x)$ , $\sE^{i}(\bb(x))=\ba(x)$ for some $i\geq 0\}$. The points $\ba(x)$, for $\ba\in \mathcal{E}(\bb,x)$, are in the forward orbit of $\bb(x)$.
		
		Given $u, d, \sE$ and $\sC$, let us fix once for all a large integer $\nu$ that satisfies
		\begin{equation*}
		\big( 1 - 4 du^2 N^n \overline{\lambda}^{\nu}\underline{\mu}^{-\nu}(1-\overline{\lambda}\underline{\mu}^{-1})^{-1} \big)^{N^n ud} > \frac{1}{2}.
		\end{equation*}
		
		The choice of $\nu$ is used in (\ref{nu}). For this value of $\nu$, let us consider $\epsilon_0=\epsilon_0(\nu)>0$ small such that $\epsilon_0 < \gamma$ and
		$$\sE^{i}(B(\bb(x), \epsilon_{0}))\cap B(\ba(x), \epsilon_{0})\neq \emptyset \text{ for some } 0\leq i \leq n+\nu \text{ only if }\ba \in \mathcal{E}(\bb,x).$$
		
		Denoting by $E_{i',j'}$ the matrix that has entry $1$ in the intersection of the $i$-th line row the $j$-th column and has all the other entries equal to $0$. For each $\ba\in I^{n}(x)$, $1\leq i' \leq u$, $1 \leq j' \leq d$, we consider a $C^\infty$ function $\phi^{\ba}_{i',j'}:\torus\rightarrow \R^d$  such that:
		\begin{itemize}
			\item $\phi^{\ba}_{i',j'}$ is supported in $B( \ba(x), \overline{\mu}^{-n} \epsilon_0)$;
			\item $D \phi^{\ba}_{i',j'} (y) = C^{-n+1}  E_{i',j'}  E^{n}$ for every $y \in B( \ba(x), \overline{\mu}^{-n} \epsilon_0/3)$;
			\item $\|D \phi^{\ba}_{i',j'} (y)\| < 2 \underline{\lambda}^{-n+1} \overline{\mu}^{n} $ for every $y\in B( \ba(x), \overline{\mu}^{-n} \epsilon_0)$.
		\end{itemize}
		
		
Define $U_x:= B(x,\epsilon_0/3)$. For every $y\in U_x$, there exists a bijection $\Phi_{x,y}:I^\infty(x)\to I^\infty(y)$ given by $\Phi_{x,y}(\ba)=\hat{\ba}=(\hat{a}_j)_{j=1}^\infty$ such that $\hat{a}_j = \pi ( h_{j,x}^{\ba}(y))$, where $h_{j,x}^{\ba}$ is the inverse branch of $E^j$ defined in (\ref{eq.S_chapeu}). Note that $[\ba]_i(y)$ is close to $[\Phi_{x,y}^{-1}(\ba)]_i(x)$ for every $\ba\in I^\infty(x)$ and every $i\geq0$. If $\pi(x)=\pi(y)$, then $\Phi_{x,y}$ is simply the identity.
		
Let $n_0$ be sufficiently large such that
$ \lfloor \frac{n_0^3}{2n_0} \rfloor < \lfloor \frac{n_0^3}{n_0+1} \rfloor -n_0 -\nu -2 $, 
which implies that $\kappa= \lfloor \frac{D}{2n} \rfloor < \lfloor \frac{D}{n+1} \rfloor - n -\nu -2$ for every $n\geq n_0$
, since $D>n^3$.
		
		\begin{claim}\label{subsequence}
			Given a sequence $\sigma=(\ba_{0},\ba_{1},\cdots,\ba_{D})$ in $ I^{\infty}(x)$ with $[\ba_{j}]_{n}$ distinct, there exists a subsequence $\hat{\sigma}=(\bb_0,\bb_1,\cdots,\bb_\kappa)$ in $I^\infty(x)$,  with $\bb_0=\ba_0$, such that $$\phi^{[\bb_{l'}]_n}_{i',j'}([\Phi_{x,y}(\bb_l)]_i(y)) = 0$$
			for every $y\in B(x,\epsilon_0)$, $1\leq l' \leq \kappa$, $0 \leq l \leq \kappa$, $l\neq l'$ and $i=0,1,\cdots,n+\nu$.
		\end{claim}
		\begin{proof}[Proof of the Claim] First, let us note that
			the cardinality of each $\mathcal{E}([\ba_j]_{n},x)$ is at most $n+1$.
Actually, if some $\ba\in I^n(x)$ is in $\mathcal{E}([\ba_j]_n,x)$ then $x$ is periodic and there must exist at most one $\ba\in I^n(x)$ such that $E^i(x)=\ba(x)$ for some $i\geq 0$. So,
			for each $0\leq i<n$ there is at most one $\ba\in I^n(x)$ such that $\ba(x)=E^{i}([\ba_j]_{n}(x))$,
			and there exists also at most one $\ba \in I^n(x)$ and one $i\geq n$ such that $\ba(x)=E^{i}([\ba_j]_{n}(x))$.
			
			
			As a consequence, there are at least $\lfloor \frac{D}{n+1}\rfloor$ elements $\{ \ba_{\xi(j)} \}$ 
			such that
$[\ba_{\xi(j')}]_n \notin \mathcal{E}([\ba_{\xi(j)}]_n,x)$ for 
				$ j'\neq j$.
				Then $\phi^{[\ba_{\xi(j')}]_n}_{i',j'}([\Phi_{x,y}(\ba_{\xi(j)})]_i(y)) =0 $
				for every $i=0,\cdots,n+\nu$ and every $y\in B(x,\epsilon_0)$
				, if  $j\neq j'$.
			
Since all the points $[a_j]_n(x)$ are distinct,  for each $j=0,1,\cdots,n+\nu$ there exists at most one $l=l(j)$ such that
				$ B([\ba_0]_j (x),\epsilon_0) \cap B([\ba_{l(j)}]_n(x),\epsilon_0)  \neq \emptyset$.
		So~we~can choose a subsequence $\sigma=(\bb_{0},\bb_{1},\cdots,\bb_{\kappa})$ such that $\bb_0=\ba_0$ and $\{\bb_j\}$ are taken among the $\{ \ba_{\xi(j)} \}$ and out of $\{\ba_{l(\hat{j})}\}_{\hat{j}=0,1,\cdots,n+\nu}$.
This choice implies that
				$\phi^{[\bb_l]_n}_{i',j'}([\Phi_{x,y}(\bb_0)]_i(y)) =0 $ for every $y\in B(x,\epsilon_0)$, $l=1,\cdots,\kappa$ and $0\leq i \leq n+\nu$.
		\end{proof}
		
Given any sequence $(\ba_0,\ba_1,\cdots,\ba_D)$  in $I^\infty(y)$ with $[\ba_i]_n$ distinct, we consider $\hat{\ba}_i:=\Phi_{x,y}^{-1}(\ba_i)$ and the sequence $\sigma_0=(\hat{\ba}_0,\hat{\ba}_1,\cdots,\hat{\ba}_D)$ in $I^\infty(x)$, which also has $[\hat{\ba}_i]_n$ distinct, and consider $\hat{\sigma}_0=(\hat{\bb}_0,\hat{\bb}_1,\cdots,\hat{\bb}_\kappa)$ the subsequence of $\sigma_0$ given by Claim \ref{subsequence}.
			Denoting $\bb_l = \Phi_{x,y}(\hat{\bb}_l) \in I^\infty(y)$ for $l=0,1,\cdots,\kappa$ and considering $\sigma=(\bb_0, \bb_1,\cdots,\bb_\kappa)$ a subsequence of $(\ba_0,\ba_1,\cdots,\ba_D)$. We will prove in the following that $\Jac(\fpsi_{y,\sigma})>\frac{1}{2}$.
		
		Denoting $\bt\in \R^s$ by $\bt=(T^{\ba})_{\ba\in I^{n}}$, where $T^{\ba}=[t^{\ba}_{i',j'}]$ is a $d\times u$ matrix and $f_{\bt}(y)=f(y)+\sum_{\ba, i',j'} t^{\ba}_{i',j'} \phi^{\ba}_{i',j'}(y)$, $\ba\in I^n(x)$, $1\leq i'\leq d$, $1\leq j'\leq u$.
		The map $\fpsi_{y,\sigma}: \R^s \to (\R^{ud})^{\kappa}$
		is an affine map of the form
		\begin{equation}
		\fpsi_{y,\sigma}(\bt)= A_{f}(y) + \sum_{\ba,i',j'} t_{\ba;i',j'} L_{\ba;i',j'}(y)
		\end{equation}
		where the $l$-th coordinate of the linear term $L_{\ba;i',j'}$, $1\leq l \leq \kappa$ is given by the series
		\begin{equation}
		[L_{\ba;i',j'}(y)]_{l}= \sum_{i=1}^{\infty}  C^{i-1} \big( D\phi^{\ba}_{i',j'} ( [\bb_{l}]_{i}(y)) - D\phi^{\ba}_{i',j'} ( [\bb_{0}]_{i}(y)) \big) E^{-i}
		\end{equation}
		
		
		Consider $W_{0}= \{ (T^{\ba})_{\ba\in I^n(x)} | T^{\ba}=0 \text{ if } \ba\neq [\hat{\bb}_{l}]_n \text{ for every } 1\leq l \leq \kappa \}$ a $\kappa u d$-dimensional subspace of $\R^{s}$. We identify a point  $\tilde{\bt} \in W_0$ with the vector $(\tilde{t}^{\bb_l}_{i',j'})$, $1\leq l \leq \kappa$, $1\leq i'\leq u$, $1\leq j'\leq d$,
		satifying
		$[T^{[\bb_l]_n}]_{i',j'}= \tilde{t}^{\bb_l}_{i',j'}$
		and
		$\tilde{t}^{\ba}_{i',j'}=0$ if $\ba$ is none of the $\bb_{l}$'s.
		
		From Claim \ref{subsequence}, for every $0\leq l \leq \kappa$, $1\leq l'\leq \kappa$ and $i=0,1,\cdots,n+\nu$, the value of $\phi^{[\bb_{l'}]_n}_{i',j'}([\Phi_{x,y}(\bb_l)]_i(y))$ is non-zero only if $l=l'$ and $i=n$.
		
		If $\tilde{t}^{\bb_l}_{i',j'}  E_{i',j'} \in W_{0}$ for some $l, i', j'$, then:
		\begin{align*}
		[L_{\bb_{l'};i',j'}(y)]_{l} = \delta_{l,l'} \tilde{t}^{\bb_l}_{i',j'} E_{i',j'}  + [R_{\bb_{l'},i',j'}]_l
		\end{align*}
		where $\delta_{l,l'}=1$ if $l=l'$ and $0$ otherwise, and
		$R_{\bb_{l'},i',j'}$ has $l$-th coordinate given by
		$$[R_{\bb_{l'},i',j'}]_l  = \sum_{i>n+\nu}^{\infty}  C^{i-1} \big( D\phi^{[\bb_{l'}]_{n}}_{i',j'} ( [\bb_{l}]_{i}(y)) - D\phi^{[\bb_{l'}]_{n}}_{i',j'} ( [\bb_{0}]_{i}(y)) \big) E^{-i}$$
		
		Each linear map $R_{\bb_{l'},i',j'}: (\R^{ud}) \rightarrow (\R^{ud})$ has the norm of supremum bounded by
		$$  u   \sum_{r> n+\nu} 4 \overline{\lambda}^{n+1}\underline{\mu}^{-n} \overline{\lambda}^{r-1}\underline{\mu}^{-r} \leq
		4 u   \sum_{r> n+\nu} \overline{\lambda}^{r+n}\underline{\mu}^{-r-n} \leq
		4u \overline{\lambda}^{\nu}\underline{\mu}^{-\nu}(1-\overline{\lambda}\underline{\mu}^{-1})^{-1}$$
		
		So the matrix $D\fpsi_{y,\sigma}{|_{W_{0}}} : W_{0} \rightarrow (\R^{ud})^{k}$ is the sum of the identity matrix with another matrix with norm bounded by:
		$$ 4\kappa du^2 \overline{\lambda}^{\nu}\underline{\mu}^{-\nu}(1-\overline{\lambda}\underline{\mu}^{-1})^{-1} \leq
		4 du^2 N^n \overline{\lambda}^{\nu}\underline{\mu}^{-\nu}(1-\overline{\lambda}\underline{\mu}^{-1})^{-1}
		=:\epsilon(\nu)$$
		
		This implies that $\Jac(\fpsi_{y,\sigma}|_{W_0}) $ is the determinant of a matrix whose entries in the diagonal are greater than $1-\epsilon(\nu)$ and the other entries have absolute value at most $\epsilon (\nu)$. By \cite{determinant}, its determinant is bounded below at least by
		\begin{equation}\label{nu}
		(1-\epsilon(\nu))^{\kappa ud} \geq
		\big( 1 - 4 du^2 N^n \overline{\lambda}^{\nu}\underline{\mu}^{-\nu}(1-\overline{\lambda}\underline{\mu}^{-1})^{-1} \big)^{N^n ud} > \frac{1}{2}
		\end{equation}
		
		Following that
		$\Jac(\fpsi_{y,\sigma}) \geq \Jac(\fpsi_{y,\sigma}|_{W_0})  > \frac{1}{2}$.
	\end{proof}
	
	\subsection{The amount of non-transversality is not too big in generic families}
In order to prove Proposition 2.6, we will first see that if $T$ does not satisfy the transversality conditions then there must be a big amount of words $\ba_i \in I^{\infty}$ that are non-transversal in some $\bc$ with distinct truncations $[\ba_i]_n$.
	
	For each $q\geq 0$, we consider an integer $p=p(q):=\lfloor\frac{-q \log\theta + \log \sqrt{d}}{\log \underline{\mu}}\rfloor + 1$ (which satisfies $\sqrt{d} \underline{\mu} ^{-p} < \theta^q$) and we consider the constant $B={\lim}_{{q\to\infty}}\frac{p(q)}{q} = \frac{\log\theta^{-1}}{\log \underline{\mu}} $. Fix
	one point $x_\bc \in \mathcal{R}(\bc)$ for each $\bc \in I^{p(q)}$. Let us fix integers $n_0, D_0, \kappa_{0}\geq 2$ such that
	\begin{align}
	(D_{0}+1)J^{-\frac{n_{0}}{2}} &<\frac{1}{2} \label{n_0}\\
	N^{\kappa_0 + B+1} \theta^{(u-d+1)\kappa_0}  &< 1   \label{kappa_0}\\
	\kappa_{0}+1  &< \frac{D_{0}}{2n_{0}}  \label{D_0}\\
	n_{0}^3  &< D_0  \label{D0}
	\end{align}
	
	There exists such integers above because $J>1$ and because the map $C$ is in $C(d;\sE)$. The choice of each one will used in the continuation:  (\ref{n_0}) is used in (\ref{D_00}), (\ref{kappa_0}) is used in (\ref{final.volume}),  (\ref{D_0}) and (\ref{D0}) are used in the proof of Proposition \ref{zero.measure} to get $\kappa=\kappa_0$ for $D=D_0$ in the definition of $n_0$-generic family.
	
	\begin{lemma}\label{tangencies}
		If $f\in\mathcal{T}$, then for every $q_0 \geq 1$ there exists $q>q_0$ such that there exists
		a word $\bc \in I^{p(q)}$  and
		$1+D_0$ words $\ba_{i}\in I^q(\bc)$ with $[\ba_i]_{n_0}$ distinct
		such that
		for any  $1\leq i \leq D_0$:
		$$m( DS(x_\bc,\ba_{i}) - DS(x_\bc,\ba_{0})  ) \leq 6  \theta^{q} \alpha_0$$
	\end{lemma}
	
	\begin{proof}
		Given $q_0$, we consider $\tilde{q}_0$ such that $\tilde{q}_0 \frac{ \log J}{ 2u\log N} >q_0$. By the transversality condition, we can take $\tilde{q}>\tilde{q}_0$ large such that  there exists
		a word $\tilde{\bc}\in I^{p(\tilde{q})}$,
		a subset $L\subset I^{\tilde{q}} (\bc)$ and
		some $\tilde{\bu}_0\in L$ such that $\# L \geq J^{\tilde{q}}$ and for every $\tilde{\bu} \in L$ there exists points $x_{\tilde{\bu}}$ and $y_{\tilde{\bu}}$ in $\mathcal{R}(\tilde{\bc})$ satisfying:
		$$ m( DS(x_{\tilde{\bu}},\tilde{\bu}) - DS(y_{\tilde{\bu}},\tilde{\bu}_0) ) \leq  3\theta^q \alpha_0 $$
		
		Fixing some $x\in \mathcal{R}(\bc)$,
		for each $d$-dimensional subspace $W$ there exists an unitary vector $v_W \in W$ such that
		$\| \big( DS(x_{\tilde{\bu}},\tilde{\bu})  - DS(y_{\tilde{\bu}},\tilde{\bu}_0) \big)v_W \| \leq 3\theta^{\tilde{q}} \alpha_0 $
		which implies that the value of $\| (DS(x,\tilde{\bu}) - DS(x,\tilde{\bu}_0)) v_W \|$ is at most:
		\begin{align*}
		\| DS(x,\tilde{\bu}) v_W - DS&(y_{\tilde{\bu}},\tilde{\bu}_0) v_W \| + \| DS(y_{\tilde{\bu}},\tilde{\bu}_0) v_W - DS(x_{\tilde{\bu}},\tilde{\bu}_0) v_W \| \\
		&+ \| DS(x_{\tilde{\bu}},\tilde{\bu}_0) v_W - DS(x,\tilde{\bu}_0) v_W \| \\
		&\leq 3\theta^{\tilde{q}} \alpha_0 +  \sqrt{d} \underline{\mu}^{-p(\tilde{q})} \alpha_0 +  \sqrt{d} \underline{\mu}^{-p(\tilde{q})} \alpha_0 \leq 5  \theta^{\tilde{q}} \alpha_0
		\end{align*}
		
		Taking the supremum on $W$, we have
		$$m( DS(x,\tilde{\bu}) - DS(x,\tilde{\bu}_0) )\leq 5  \theta^{\tilde{q}} \alpha_0 $$
		
		To obtain words $\ba_i$ with distinct truncations $[\ba_i]_{n_0}$,
		for each $0\leq j \leq \lfloor\tilde{q}/n_{0}\rfloor$ and $\ba \in L$, we define
		$$L(j,\ba) = \{ \bb\in L, [\ba]_{jn_0}=[\bb]_{jn_0} \}$$
		if $j\geq 1$ and $L(0,\ba)=L$.
		
		Note that $L(j_1,\ba) \supset L(j_2,\ba)$ if $j_1\leq j_2$ and that  $L(j,\ba_1) \cap L(j,\ba_2) = \emptyset$ if $[\ba_1]_{jn_0} \neq [\ba_2]_{jn_0}$.
		Define
		$$H(j,\ba):= \{\bb\in L | \bb \in L(j,\ba) \cap  L(j-1,\bu_0)\} \text{ and   } h_{j}:=\underset{a\in L}{\max} \#H(j,\ba) $$
		for $j\geq 1$ and $h_{0}=\# L \geq J^{\tilde{q}}$.
		
		Since $h_{j}\leq N^{u(\tilde{q}-j n_0)}$ , there exists $j\leq\lfloor\tilde{q}/n_0\rfloor$ such that $h_{j+1}< J^{- n_0/ 2} h_{j}$. Let $j_{*}$ be the minimum of such integers $j$ and put $q=\tilde{q} - n_0 j_{*}$. By the minimality of $j_*$ we have that $h_{j_{*}}\geq J^q $. We also have that $q\geq \tilde{q} \frac{\log J}{2u\log N} > q_0 $.
		
		Considering $\bv_0$ such that $h_{j_*} = \# H(j_*,\bv_0)$, the set $H(j_*,\bv_0)$ contains at least $1+D_0$ sets $H(j_*+1,\tilde{\bu}_i)$, that is, there exists $\tilde{\bu}_1,\cdots, \tilde{\bu}_{D_0} $ such that $[\tilde{\bu}_{i'}]_{(j_*+1)n_0}$ are distinct for  $i'=0,1,\cdots,D_0$,
		because
		\begin{align}\label{D_00}
		h_{j_{*}} - (D_0 +1) h_{j_{*}+1} > h_{j_{*}}-(D_0 +1)J^{- n_0/2} h_{j_{*}}>0
		\end{align}
		
		So, there exists $\bb\in I^{\tilde{q}-q}$ and $\ba_{i}\in I^{q}$, $0\leq i \leq D_0 $ such that $\bb\ba_{i}=\tilde{\bu}_i \in H(j_*,\bv_0)$ for $0\leq i \leq D_0  $  and that $[\ba_{i}]_{n_0}\neq [\ba_{j}]_{n_0}$ if $i\neq j$.
		
		Finally, from
		\begin{equation}
		DS(x,\bb\ba_i)=DS(x,\bb)+C^{\tilde{q}-q}DS(\bb(x),\ba_i)E^{-\tilde{q}+q}
		\end{equation}
		we have that:
		$$ m( DS(\bb(x),\ba_{i}) - DS(\bb(x),\ba_{0}) ) \leq   5  \theta^{q} \alpha_0 $$
		
		Taking $\bc\in I^{p(q)}$ such that $\bb(x)\in\mathcal{R}(\bc)$, we have  $\|\bb(x) - x_c \|\leq \sqrt{d}  \underline{\mu}^{-p(q)} \leq  \theta^q$, which implies the Lemma.
	\end{proof}
	
	To finish the proof of Proposition \ref{zero.measure}, we will use two lemmas of linear algebra.
	
	\begin{lemma}\label{jacobian}
		Given integers $m$ and $k$, there exists a constant $C_3>0$ such that if  $G:\R^s \to \R^k$ is an affine function with  $\Jac(\fpsi_{x,\sigma})>\delta$, then
		\begin{equation}
		m_s (G^{-1}(Y)\cap [-1,1]^s) \leq C_3 \frac{m_k(Y)}{\delta}
		\end{equation}
		for every measurable set $Y\subset \R^k$.
	\end{lemma}
	\begin{proof}
		Immediate from the Jacobian of an affine map.
	\end{proof}
	
	\begin{lemma}\label{volume}
		For every $u\geq d$, there exists a constant $C_4$  such that the set $$\mathcal{X}(r):=\{ M \in M(d \times u ) , \|M\| \leq 2\alpha_0 , m(M) <r \}$$ has volume bounded by $C_4 r^{u-d+1}$ for every $r>0$.\footnote{The authors thanks A. Quas for pointing this Lemma.}
	\end{lemma}
	\begin{proof}
		Let $r(i)$ be the $i$-th row of the matrix $M$. First, we claim that $\fm(M)$ is equals, up to a bounded factor, to $\min_i {d\big( r(i) , {\operatorname{span}}_{j\neq i} (r(j) \big)}$.
		
		\begin{claim}\label{bounded.factor} For every $M\in M(d\times u)$, it is valid that:
			$$\displaystyle \fm(M) \leq \min_i {d\Big( r(i) , \underset{j\neq i}{\operatorname{span}} (r(j) \Big)} \leq \sqrt{d} \fm(M)$$
		\end{claim}
		\begin{proof}[Proof Claim \ref{bounded.factor}]
			Consider $D:=\min_i {d\big( r(i) , \operatorname{span}_{j\neq i} (r(j) \big)}$. Remind that $\fm(M)=\fm(M^T)$ and  that $r(i)=M^T(e_i)$, where $M^T$ is the transpose matrix of $M$ and $\{e_i\}_{i=1,\cdots,d}$ the canonical basis of $(\R^d)^*$. Considering $v=(x_1,\cdots,x_d) \in (\R^d)^*$ an unitary vector such that $\|M^T(v)\|=\fm(M^T)=\fm(M)$ and $i_1$ such that $|x_{i_1}|=\max_j{|x_j|}$, we have that $|x_{i_1}|\geq\frac{1}{\sqrt{d}}$. Then
			\begin{align*}
			\sqrt{d} \fm(M) \geq \frac{\|M^T(v)\|}{|x_{i_1}|} 
			= \Big\| \frac{x_{i_1}}{|x_{i_1}|} r({i_1}) + \sum_{j\neq {i_1}} \frac{x_j}{|x_{i_1}|} r(j)\Big\| \geq d\Big(r({i_1}), \underset{j\neq {i_1}}{\operatorname{span}} (r(j)) \Big)  \geq D.
			\end{align*}
			
			On the other hand, let $i_{2}$ be such that $d\big(r(i_{2}), {\operatorname{span}}_{j\neq i_{2}} (r(j)) \big) = D$. There exists real numbers $b_j$, $j\neq i$, such that $D = \| r(i_{2}) + \sum_{j\neq i} b_j r(j) \|$. Then
			\begin{align*}
			D = \| M^T( e_{i_{2}} + \sum_{j\neq i} b_j e_j ) \| \geq \fm(M^T) \| e_{i_{2}} + \sum_{j\neq i} b_j e_j \| \geq \fm(M)
			\end{align*}
		\end{proof}
		
		For every choice of $d-1$ rows of $M$, their span is a $(d-1)$-dimensional subspace and its $r$-neighborhood has volume bounded by $C r^{u-(d-1)}$. Looking to each matrix $M$ as vector formed by its rows $(r(i))_{i=1,\cdots,d} \in (\R^u)^d$, we see that the set $\mathcal{X}(r)$ is contained in the finite union of the set of matrices where the row $r(i)$ is the $r$-neighborhood of the span of the other $d-1$ rows, following that the volume of $\mathcal{X}(r)$ is bounded above by $C_4 r^{u-d+1}$.
	\end{proof}
	
	\begin{proof}[Proof of Proposition \ref{zero.measure}]
		We consider $\{ \phi_{i}\}_{i=1}^{s}$ the functions given by Proposition \ref{generic.families} for $E$, $C$ and $n=n_{0}$.
		
		Fixing words of infinite length $\ba^1,\cdots,\ba^r \in I^{\infty}$ with $[\ba^i]_1 = i$, we associate for every word $\ba$ of finite length a word $\hat{\ba}=\ba \ba^i \in I^\infty$. For any sequence $\sigma=(\bb_i )_{i=0}^{\kappa_0}$ in $(I^{q})^{1+\kappa_0}$, denote $\hat{\sigma}=(\hat{\bb}_i )_{i=0}^{\kappa_0} \in I^\infty(\bc)$.
		
		Let us consider $\mathcal{T}_0 := \{ \bt \in \R^s | f_\bt \in \mathcal{T} \}$.
		If $\bt\in \mathcal{T}_0$,  Lemma \ref{tangencies} implies that for every integer $q_0$ there exists an integer $q\geq q_0$,
		a word $\bc \in I^{p(q)}$ and
		$1+D_0$ words $\ba_{i}\in I^q(\bc)$ with $[\ba_i]_{n_0}$ distinct
		such that $\fm( DS(x_\bc,\ba_{i}) - DS(x_\bc,\ba_{0})  ) \leq 6  \theta^{q} \alpha_0$
		for any  $i =1,\cdots, D_0$. In particular, it holds
		\begin{equation}
		\fm( DS(x_\bc,\hat{\ba}_{i}) - DS(x_\bc,\hat{\ba}_{0})  ) \leq 7  \theta^{q} \alpha_0.
		\end{equation}
		
We consider also the sets $\mathcal{B}^{q} = \{ (\sigma, \bc) \in (I^{q})^{1+\kappa_0} \times I^{p(q)} | \Jac(\fpsi_{x_{c},\hat{\sigma}})>\frac{1}{2} \}$ and  $\mathcal{T}(q):={\cup}_{(\sigma,\bc)\in\mathcal{B}^q} \fpsi_{x_\bc, \hat{\sigma}}^{-1}( \mathcal{X}(7 \theta^q \alpha_0)^{\kappa_0})$.
		
		Given $x_\bc$ and the sequence $(\hat{\ba}_{0},\hat{\ba}_{1},\cdots, \hat{\ba}_{D_0})$, the facts that  the family $T_{\bt}^{n_0}$ is $n_0$-generic,  that
			$\frac{D_0}{2n_0}>\kappa_0 $ and $D_0 > n_0^3$, it
		implies  that
		there exists a subsequence $\sigma=(\bb_0,\bb_1,\cdots,\bb_\kappa)\in (I^q)^{1+\kappa_0}$ such that
		each entry of $\fpsi_{x_\bc, \hat{\sigma}}(\bt)$, for $\bt\in \mathcal{T}_0$, is in the set
		$$\mathcal{X}(7\theta^q \alpha_0)= \{ M\in M(d\times u), \|M\| \leq 2 \alpha_0 \text{ and } m(A) < 7 \theta^q \alpha_0 \},$$
		what means that  $\mathcal{T}_0 \subset {\liminf}_{q\to\infty} \mathcal{T}(q)$.
		
		Since for every $i\in\overline{I}$ there are exactly $N$ sets $\mathcal{R}(j)$ such that $E(\mathcal{R}(j)) \cap \mathcal{R}(i) \neq \emptyset$, we get that
		$\# I^q = r N^{q-1}$
		and that
		\begin{equation}\label{cardinality}
		\#\mathcal{B}^q \leq r (r N^{(q-1)})^{(\kappa_0 +1)} (r N^{p(q)-1})\,.
		\end{equation}
		
		Putting together  Lemma \ref{jacobian}, Lemma \ref{volume} and (\ref{cardinality}), we get that the estimate
		\begin{align}\label{final.volume}
		m_s(\mathcal{T}(q) ) \leq  r^{3+\kappa_0} N^{q+q\kappa_0  + p(q)-2 -\kappa_0}   \big( (14C_4 \theta^q \alpha_0 )^{u-d+1}\big)^{\kappa_0}
		\end{align}
		is valid for infinitely many $q$'s.
		
		So, there is a constant $C_5>0$ such that the term in (\ref{final.volume}) is bounded above by
		\begin{equation}
		C_5 ( N^{\kappa_0 + B+1}  \theta^{(u-d+1)\kappa_0} )^q
		\end{equation}
		when $q$ is sufficiently large. By the choice of $\kappa_0$, it converges to zero exponentially  fast when $q \to +\infty$, implying the conclusion of Proposition \ref{zero.measure}.
	\end{proof}
	
	Finally, we are able to prove Theorem B and Corollary C
	
	\begin{proof}[Proof of Theorem B]
		Given the map $T=T(E,C,f)$ satisfying the assumptions of Theorem B, we consider
		the functions $\{\phi_k\}_{k=1}^s \in C^\infty(\torus,\R^d)$  given by Proposition \ref{generic.families} for some $n$ sufficiently large. Then the set of parameters $\bt=(t_1,\dots,t_s)$
		for which the corresponding  map $T_\bt$ satisfies the transversality condition has full Lebesgue measure. Theorem D implies that for such maps the
		SRB measure $\mu_{T_\bt^n}$ is absolutely continuous. Finally, Theorem B follows just noting that the SRB measure of $T_\bt^n$ is the same of $T_\bt$.
	\end{proof}
	
	\begin{proof}[Proof of Corollary C]
		When $d=1$, the map $C$ is just a multiplication by a factor $\lambda \in \R$ and if $E=\mu I$ for some integer $\mu\geq 2$ then the relations in the definition of $C(d,E)$  become $\lambda\in(\frac{1}{\mu^u},1)$. So we are under the assumptions of Theorem B.
	\end{proof}
	
\section*{Acknowledgments}

The authors thank Professor Marcelo Viana for his useful discussions and IMPA for its hospitality during the Summer Program.


\begin{thebibliography}{A}
\bibitem{Alves} J. F. Alves - \textit{SRB measures for non-hyperbolic systems with multidimensional expansion}, Ann. Sci. Ecole Norm. Sup., 33(4) (2000), 1–32.
		
\bibitem{fat.baker.map} J. C. Alexander, J. A. Yorke -  \textit{Fat Baker's transformations}, Ergod. Th. Dynam. Sys., 4 (1984), 1-23.
	
\bibitem{tsujii.ssd} A. Avila, S. Gouezel, M. Tsujii - \textit{Smoothness of Solenoidal Attractors}, Discr. and Cont. Dyn. Sys., 15 (2006), 21-35.
		
\bibitem{bortolotti} R. T. Bortolotti - \textit{Physical measures for certain partially hyperbolic attractors on 3-manifolds}, arXiv preprint (2015).		
		
\bibitem{bowen} R. Bowen - \textit{Equilibrium states and the ergodic theory of Anosov diffeomorphisms}, Lect. Notes in Math., Springer-Verlag, 470 (1975).
		
\bibitem{markov.partition} M. Craizer - \textit{Teoria Erg\'odica das Transforma\c{c}\~oes Expansoras}, MSc dissertation, IMPA, (1985).
		
\bibitem{keller} G. Keller -  \textit{Exponents, Attractors, and Hopf decompositions for interval maps}, Ergod. Th. Dynam. Sys., 10 (1990), 717-744.

\bibitem{palis} J. Palis -  \textit{A global view of dynamics and a conjecture on the denseness of finitude of attractors}, Ast\'erisque, 261 (2000), 335-347.
		
\bibitem{determinant} A. M. Ostrowski - \textit{Sur la d\'etermination des bornes inf\'erieures pour une classe des d\'eterminants}, Bull. Sci. Math., 61(2) (1937), 19–32.
		
\bibitem{solano} J. Solano - \textit{Non-uniform hyperbolicity and existence of absolutely continuous invariant measures}, Bull. Braz. Math. Soc., 44(1) (2013), 67-103.
		
\bibitem{tsujii.fsa} M. Tsujii - \textit{Fat Solenoidal Attractors}, Nonlinearity, 14 (2001), 1011-1027.

\bibitem{tsujii.acta} M. Tsujii - \textit{Physical measures for partially hyperbolic surface endomorphisms}, Acta Math., 194 (2005), 37-132.		
\end{thebibliography}
\end{document}